\documentclass[preprint,11pt]{elsarticle}

\usepackage{amsfonts, amsmath, amscd}
\usepackage[psamsfonts]{amssymb}

\usepackage{amssymb}

\usepackage{url}
\usepackage{pb-diagram}

\usepackage[all,cmtip]{xy}

\usepackage{natbib}

\usepackage[usenames]{color}

\newtheorem{theorem}{Theorem}[section]

\newtheorem{corollary}[theorem]{Corollary}

\newtheorem{remark}[theorem]{Remark}

\newtheorem{example}[theorem]{Example}

\newproof{proof}{Proof}

\numberwithin{equation}{section}
\numberwithin{theorem}{section}

\input xy
\xyoption{all}

\newcommand{\NN}{\mathbb{N}}

\newcommand{\w}{\omega}

\newcommand{\IR}{\mathbb{R}}

\renewcommand{\phi}{\varphi}

\newcommand{\spn}{\mathrm{span}}

\begin{document}

\begin{frontmatter}

\title{A topological group observation on \\ the Banach--Mazur separable quotient problem}

\author{Saak S. Gabriyelyan}
\ead{saak@math.bgu.ac.il}
\address{Department of Mathematics, Ben-Gurion University of the Negev, Beer-Sheva, P.O. 653, Israel}

\author{Sidney A. Morris}
\ead{morris.sidney@gmail.com}
\address{Faculty of Science and Technology, Federation University Australia, PO Box 663, Ballarat, Victoria, 3353,  Australia \& Department of Mathematics and Statistics,  La~Trobe University, Melbourne, Victoria, 3086, Australia}

\begin{abstract}
\end{abstract}


\begin{abstract}
The Banach-Mazur problem, which asks if every infinite-dimensional  Banach space has an infinite-dimensional  separable quotient space, has remained unsolved for 85 years,  but has been answered in the affirmative for special cases such as reflexive Banach spaces.  It is also known that every infinite-dimensional non-normable Fr\'{e}chet space has an infinite-dimensional  separable quotient space, namely
$\IR^\w $. It is proved in this paper that every infinite-dimensional
Fr\'{e}chet space (including every infinite-dimensional Banach space), indeed every locally convex space which has a subspace which is an infinite-dimensional Fr\'{e}chet space, has an infinite-dimensional (in the topological sense) separable metrizable quotient group, namely $\mathbb{T}^\w$,  where $\mathbb{T}$ denotes the compact unit circle group.

\end{abstract}
\end{frontmatter}


\section{Introduction}


The famous problem of Stefan Banach and Stanis\l aw Mazur in the 1930s asks if every (real) Banach space has a separable infinite-dimensional quotient space. This has been shown to be true for reflexive Banach spaces \cite{Pelczynski} and more generally Banach spaces which are dual spaces
 \cite{Argyros}, and weakly compactly-generated Banach spaces \cite{AmirLindenstrauss}, with further results in \cite{KakolSaxon, KakolSaxonTodd, KakolSliwa, SaxonN, SaxonWilansky}, but the general question remains open. However, M. Eidelheit \cite{Eidelheit} proved that every non-normable Fr\'{e}chet space has a separable quotient locally convex  space, namely $\IR^\w$. We note that $\IR^\w$ cannot be a quotient locally convex space (or even a quotient group) of a Banach space, but every separable Banach space is indeed homeomorphic to $\IR^\w$, see \cite{Bessaga}.

 Noting that $\mathbb{T}^\w$ is a quotient group of $\IR^\w$, where $\mathbb{T}$ is the compact unit circle group, we see that every non-normable Fr\'{e}chet space has $\mathbb{T}^\w$ as a quotient group. This leads us to ask if every Banach space also has $\mathbb{T}^\w$ as a quotient group.  In fact we prove this and more in the theorem below.

\section{Results}

\begin{theorem} \label{t:Banach-T^N}
Let $E$ be a (Hausdorff) locally convex space over the field ${\mathbf F}$, where ${\mathbf F}=\IR$ or  $\mathbb{C}$. If $E$ has a subspace which is an infinite-dimensional Fr\'{e}chet space, then $E$ has $\mathbb{T}^\w$ as a quotient group.
\end{theorem}

\begin{proof}
By \cite[Theorem~2]{Sliwa}, $F$ has a basic sequence $\{ e_n: n\in\w\}$, that is, each element of the closure $H$ of the span  of $\{ e_n: n\in\w\}$ can be written uniquely as $x=\sum_{n\in\w} a_n e_n$ with $a_n\in  \mathbf{F}$  and the coefficient functionals $\chi'_n: H\to\mathbf{F}, \chi'_n(x):= a_n,$ are continuous. For every $n\in\w$, denote by $\chi_n$ some extension of $\chi_n$ onto the whole space $E$. For every
 $n\in\w$, choose $m_n\in\w$ such that $e_n/m_n \to 0$ in $E$ (such numbers exist since $F$ is metrizable) and set $f_n:= \mathrm{Re}(\chi_n)$ (so $f_n$ is a real continuous linear functional of $E$).
 Define a homomorphism $R: E\to \mathbb{T}^\w$ by
\[
R(x) := \left( e^{2\pi i \, 2^n m_n f_n(x)} \right), \quad x\in E.
\]
Let us show that $R|_H$ is surjective. Indeed, fix arbitrarily $z=\left( e^{2\pi i a_n} \right)\in \mathbb{T}^\w$, where $a_n \in [0,1)$ for every $n\in\w$. For every $n\in\w$, set $b_n:= a_n /(2^n m_n)$ and $x_n=\sum_{i=0}^{n} b_n e_n$. Then, for every $N< n<k$ and each seminorm $p$ on the Fr\'{e}chet space $F$, we have
\[
p(x_k-x_n)\leq \sum_{i=n+1}^{k} \frac{a_n}{2^n} p\big( e_n/m_n\big) \to 0 \quad \mbox{ at }  N\to \infty.
\]
Therefore the sequence $\{ x_n\}$ is Cauchy, and hence it converges to $x=\sum_{n\in\w} b_n e_n \in F$. Since all the numbers $b_n$ are real, it follows that $R(x)=z$. Thus $R|_H$ is surjective.

Since the group $\mathbb{T}^\w$  carries the product topology and all the homomorphisms $x \mapsto e^{2\pi i \, 2^n f_n(x)}\in \mathbb{T}$ are continuous, we obtain that $R$ is continuous. Therefore by the Open Mapping Theorem,  Theorem 1.2.6 of \cite{Beck-Kechris}, the surjective continuous homomorphism $R|_H$ from the separable Fr\'{e}chet space $H$ onto the compact metrizable group $\mathbb{T}^\w$ is open.
Thus the map $R$ also is open. \qed
\end{proof}

\begin{corollary} \label{c:factor-Frechet}
Every infinite-dimensional Fr\'{e}chet space, and in particular every infinite-dimensional Banach space, has $\mathbb{T}^\w$ as a quotient group.
\end{corollary}

\begin{remark} {\em
We note that by Theorem 8.4.6 of \cite{Hofmann-Morris} and Corollary \ref{c:factor-Frechet}: a compact group $G$ is a quotient group of an infinite-dimensional Fr\'{e}chet space if and only if $G$ is topologically isomorphic to $\mathbb{T}^N$ for some $N\leq \aleph_0$.}
\end{remark}

Example \ref{exa:Banach-1} shows the the condition in Theorem \ref{t:Banach-T^N} is not necessary, but Example \ref{exa:Banach-2} shows that the condition cannot be dropped entirely since there is a complete locally convex space which does not have $\mathbb{T}^\w$ as a quotient group.
\begin{example} \label{exa:Banach-1}
Let $(E,\tau_E)$ be the Banach space $\ell_\infty$ endowed with the topology $\tau_E$ induced from the complete metrizable space $\IR^\w$. Then (1) $E$ does not contain an infinite-dimensional Fr\'{e}chet subspace; and (2)  $E$ has $\mathbb{T}^\w$ as a quotient group.
\end{example}

\begin{proof}
(1) Let  $(F,\tau_F)$ be a Fr\'{e}chet subspace of $(E,\tau_E)$. Then it is a subspace of $\IR^\w$. Hence  $F$  with the weak topology denoted by $F_w$ is also a subspace of $\IR^\w$. Therefore $F_w$ is metrizable. On the other hand, since $F=\bigcup_{n\in\w} (F\cap nB)$, where $B$ is the closed unit ball of $\ell_\infty$, the Baire category theorem implies that $F\cap B$ is a neighborhood of zero in $F$ (note that $B$ is closed in $(E,\tau_E)$). Therefore the topology $\tau_F$ of $F$ is finer than the norm topology $\tau_\infty$ induced from $\ell_\infty$. Clearly, we also have $\tau_F\leq \tau_\infty$, and hence $(F,\tau_F)$ is a subspace of the Banach space $\ell_\infty$. Thus $(F,\tau_F)$ is a Banach space. But a Banach space in the weak topology is metrizable if and only if it is finite-dimensional, see for example Theorem 1.5 of \cite{GKP}. Thus $F$ is finite-dimensional.
 
(2) To show that $(E,\tau_E)$ has $\mathbb{T}^\w$ as a quotient group consider the map
\[
T: E\to \mathbb{T}^\w, \quad T(x_n)=\big( e^{2\pi i x_n}\big) \; \mbox{ for } \; (x_n)\in E.
\]
It is clear that $T$ is a surjective and open continuous homomorphism. \qed
\end{proof}

Denote by $\phi$ the complete countably infinite-dimensional locally convex space which is the strong dual space of $\IR^\w$. We note that $\phi$ is the inductive limit of $\IR^n$s.

\begin{example} \label{exa:Banach-2}
There is no continuous surjective homomorphism from $\phi$ onto $\mathbb{T}^\w$.
\end{example}

\begin{proof}
Let $T:\phi \to \mathbb{T}^\w$ be any continuous homomorphism. We have to show that $T$ is not onto. For every $n\geq 1$, denote by $T_n$ the restriction of $T$ to the closed vector subspace $\IR^n$ of $\phi$. Then $T_n$ induces a continuous monomorphism ${S_n}$ from $\IR^n/\ker(T_n)$ to $\mathbb{T}^\w$. By Theorem 9.11 of \cite{HR1} this implies that the group $\IR^n/\ker(T_n)$ is topologically isomorphic to $\IR^{k_n} \times \mathbb{T}^{s_n}$ for some integers $k_n$ and $s_n$. Hence ${S}_n : \IR^{k_n} \times \mathbb{T}^{s_n} \to \mathbb{T}^\w $ is a continuous injective homomorphism. It follows that 
\begin{equation} \label{equ:Banach-Problem-1}
T(\phi)= \bigcup_{n\geq 1} \left( \bigcup_{p\in\NN} {S}_n \big( [-p,p]^{k_n} \times \mathbb{T}^{s_n} \big) \right).
\end{equation}
Since ${S}_n$ is a homeomorphism on the compact finite-dimensional  topological space $[-p,p]^{k_n} \times \mathbb{T}^{s_n}$, the equality (\ref{equ:Banach-Problem-1}) implies that $T(\phi)$ is countable dimensional topological space. However, the metrizable group $ \mathbb{T}^\w$ is not a countable dimensional topological space by Corollary 3.13.6 of \cite{vanMill}. Thus $T$ is not surjective.\qed
\end{proof}

\begin{remark}{\em
We mention that A. Leiderman, M. Tkachenko and the second author in their paper \cite{LeidermanMorrisTkachenko} have examined the question: which topological groups $G$ have separable quotient groups? Their results include substantial information on the cases that $G$ is (i) a compact group, (ii) a pro-Lie group, (iii) a pseudocompact group and (iv) a precompact group. }
\end{remark}

\noindent\textbf{Acknowledgement} The authors thank Karl Heinrich Hofmann for useful comments. The second author thanks Ben Gurion University of the Negev for hospitality during his visit.


\section*{References}

\begin{thebibliography}{99}




\bibitem{AmirLindenstrauss} D.~Amir, J.~Lindenstrauss,
\newblock The structure of weakly compact sets in Banach spaces,
\newblock \textit{Ann. of Math.} \textbf{88} (1968), 33--46.

\bibitem{Argyros} S.A.~Argyros, P.~Dodos, V.~Kanellopoulos,
\newblock Unconditional families in Banach spaces,
\newblock \textit{Math. Ann.} \textbf{341} (2008), 15--38.


\bibitem{Beck-Kechris}
H. Becker, A.S. Kechris, The Descriptive Set Theory of Polish Group Actions, Cambridge University Press, Cambridge, 1996.

\bibitem{Bessaga}
C. Bessaga, A.A. Pe\l czy\'nski,
\newblock Selected Topics in Infinite-Dimensional Topology,
\newblock Polish Scientific Publishers, Warszawa, Poland, 1975.


\bibitem{Eidelheit} M. Eidelheit,
\newblock Zur Theorie der Systeme linearer Gleichungen,
\newblock \textit{Studia Math.} \textbf{6} (1936), 130--148.


\bibitem{GKP}
S.~Gabriyelyan, J.~K{\c{a}}kol,  G. Plebanek, The Ascoli property for function spaces and the weak topology of Banach and Fr\'echet spaces, Studia Math. \textbf{233} (2016), 119--139.

\bibitem{HR1}
E.~Hewitt, K.A.~Ross,  \emph{Abstract Harmonic Analysis}, Vol. I, 2nd ed. Springer-Verlag, Berlin, 1979.


\bibitem{Hofmann-Morris}
K.H. Hofmann, S.A. Morris, \emph{The structure of compact groups},  3rd ed. De Gruyter, Berlin, 2013.


\bibitem{KakolSaxon} J. K{\c{a}}kol, S.A. Saxon,
\newblock Separable quotients in $ C_{c}( X)$, $ C_{p}( X) $, and their duals,
\newblock \textit{Proc. Amer. Math. Soc.} \textbf{145} (2017), 3829--3841.

\bibitem{KakolSaxonTodd}  J. K{\c{a}}kol, S.A. Saxon, A. Todd,
\newblock Barrelled spaces with(out) separable quotients,
\newblock \textit{Bull. Austral. Math. Soc.} \textbf{90} (2014), 295--303.


\bibitem{KakolSliwa}  J. K{\c{a}}kol, W. \'{S}liwa,
\newblock Remarks concerning the separable quotient problem,
\newblock \textit{Note di Matematica} \textbf{13} (1993), 277--282.

\bibitem{LeidermanMorrisTkachenko} A.G. Leiderman, S.A. Morris,  M.G. Tkachenko,
\newblock The separable quotient problem for topological groups,
\newblock \url{https://arxiv.org/abs/1707.09546}.

\bibitem{Pelczynski} A.~Pe\l czy\'nski,
\newblock Some problems on bases in Banach and Fr\'echet spaces,
\newblock \textit{Israel J. Math.} \textbf{2} (1964), 132--138.



\bibitem{vanMill}
J. van Mill, The Infinite-Dimensional Topology of Function Spaces, Elsevier, North-Holland, Amsterdam, 2006.


\bibitem{SaxonN} S.A.~Saxon, P.P.~Narayanaswami,
\newblock Metrizable (LF)-spaces, (db)-Spaces, and the separable quotient problem,
\newblock \textit{Bull. Austral. Math. Soc.} \textbf{23} (1981), 65--80.


\bibitem{SaxonWilansky} S.~Saxon, A.~Wilansky,
\newblock The equivalence of some Banach space problems, 
\newblock \textit{Colloq. Math.} \textbf{37} (1977), 217--226.


\bibitem{Sliwa}
 W. \'{S}liwa,
 \newblock Every infinite-dimensional non-archimedian Fr\'{e}chet space has an orthogonal basic sequence, \textit{Indag. Mathem.} \textbf{11} (2000), 463--466.


\end{thebibliography}
\end{document}